\documentclass[a4paper,10pt]{amsart}

\usepackage{amsmath}
\usepackage{amssymb}
\usepackage{graphicx}

\newcommand{\A}{\parallel}

\usepackage{amsthm}
\theoremstyle{definition}
\newtheorem*{dfn1} {Definition    1}
\newtheorem*{dfn2}{Definition     2}
\newtheorem*{dfn3} {Definition    3}
\newtheorem*{exa1}{Example 1}
\newtheorem*{exa2}{Example 2}
\newtheorem*{exa3}{Example 3}

\newtheorem*{thm1}{ Theorem      1}
\newtheorem*{thm2}{Theorem       2}
\newtheorem*{prop1}{Proposition  1}

\newtheorem*{pro3}{Proposition   3}
\newtheorem*{pro4}{Proposition   4}
\newtheorem*{pro5}{Proposition   5}
\newtheorem*{pro6}{Proposition    6}
\newtheorem*{lem1}{Lemma      1}
\newtheorem*{lem2} {Lemma       2                  (The zero divisor version of the Tietze extension theorem)}
\newtheorem*{lem3} {Lemma       3}
\newtheorem*{lem4}  {Lemma       4}

\newtheorem*{remark3}{Remark  3}

\newtheorem*{remark5}{Remark  5}
\newtheorem*{remark6}{Remark  6}
\newtheorem*{remark7}{Remark  7}

\newtheorem*{proof1}{Proof  of theorem 1}

\newtheorem*{question1}{Question   1}
\newtheorem*{question2}{Question   2}
\newtheorem*{question3}{Question   3}
\newtheorem*{question4}{Question   4}
\newtheorem*{question5} {Question  5}
\newtheorem*{question6} {Question  6}
\newtheorem*{question7} {Question  7}

\theoremstyle{plain}

\newtheorem*{coro1}{Corallary  1}
\newtheorem*{coro2}{Corallary 2 }
\newtheorem*{coro3}{Corollary  3}
\newtheorem*{coro4}{Corollary  4}
\newtheorem*{coro}{Corollary   5}

\begin{document}

\title[Positive Zero divisors in $C^{*}$ algebras]{A Note on  positive zero  divisors in $C^{*}$ algebras}
\author{Ali Taghavi}

\address{Faculty of Mathematics and Computer Science,  Damghan  University,  Damghan,  Iran.}
\email{taghavi@du.ac.ir}

\date{\today}


\begin{abstract}
  In this paper we concern with  positive zero  divisors  in  $C^{*}$  algebras. By means of zero divisors, we introduce a hereditary invariant for $C^{*}$ algebras. Using this invariant, we give an example of a $C^{*}$ algebra $A$ and a $C^{*}$ sub algebra $B$ of $A$ such that there is no a hereditary imbedding of $B$ into $A$.\\
   We also introduce a new concept zero divisor real rank of a $C^{*}$ algebra, as  a zero divisor analogy of real rank theory of $C^{*}$ algebras. We observe that this quantity is zero for $A=C(X)$  when $X$ is  a separable compact Hausdorff space or $X$ is homeomorphic to the unit square with the lexicographic topology.    To a $C^{*}$  algebra $A$ with $\dim A > 1$, we  assign the undirected graph $\Gamma^{+} (A)$ of  non  zero positive  zero divisors. For the Calkin algebra $A=B(H)/K(H)$, we show  that $\Gamma^{+}(A)$  is a connected graph and diam $\Gamma^{+}(A)= 3$. We show that $\Gamma^{+}(A)$ is  a connected graph  with  $\text {diam}\; \Gamma^{+}(A)\leq 4$, if $A$ is  a factor.
\end{abstract}

\maketitle

\section*{introduction}
The  zero  divisor problem  is  an important problem in the theory of group algebras. For a torsion free group $\Gamma$, it is conjectured that the  group  algebra $\mathbb{C} \Gamma$ does not have nontrivial  zero divisor. That is $0$ is the  only zero divisor in this algebra. In parallel to this problem, there is  another conjecture, so called idempotent conjecture, which asserts  that $0$ and $1$ are the only idempotents in $\mathbb{C} \Gamma$. For  more information on these two problems  see  \cite{VALETTE}.\\
One  can consider these two problems in the context of $C^{*}$ algebras. For  a  general  $C^{*}$ algebra, the idempotent  problem is equivalent to projection problem.  Recall that a projection is a self adjoint idempotent element
, that is  an element $a$ with $a=a^{*}=a^{2}$. For  equivalency of  projection and idempotent problem in $C^{*}$ algebras, see \cite[page 101]{OLSEN}.  \\

A  $C^{*}$  algebra  may or may not have an idempotent. While any  $C^{*}$  algebra $A$, which is  not a one  dimensional algebra has always positive zero  divisor. In fact for a self adjoint element $a\in A$ we have $a_{-}a_{+}=0$ where $a_{-}$ and $a_{+}$ are negative and positive part of $a$. It can be easily shown that there is a self adjoint element$a\in A$ such that $a_{-}$ and $a_{+}$ are non zero elements provided $A$ is not a one dimensional algebra. Because of this triviality, zero divisors  are  not counted as important objects in $C^{*}$ algebras. However they are counted as important objects in group algebras.
  Notwithstanding of this triviality, in this paper we focus on zero divisors in $C^{*}$ algebras in order to obtain some new invariants for $C^{*}$ algebras. To start, using an interesting result in  \cite[proposition 2.5]{WOLF},  the proposition below,    we give  an alternative proof  for existence of zero divisors.\\
\begin{prop1}
Let  $U$,  $V$  be  unital  *-homomorphisms  from  the  unital  $C^{*}$-algebra  $A$
into  the unital C*-algebra $B$.  If  $ab=  0$ implies $(U a)(Vb)  = 0$  for all self  adjoint  elements $a$ and $b$  in $A$  then $U  =  V$
\end{prop1}

Let $a$  be  a projection. Put $b=1-a$. Then $a$ and $b$ are  positive elements with $ab=0,\;|a|=|b|=1,\;a+b=1$. So existence of projection implies that we have positive zero divisors. Of course the  converse is not true, since there are $C^{*}$ algebras without nontrivial projections.  Now we try to introduce a reasonable converse as follows:

 For  a  $C^{*}$ algebra  $A$  we  define a real  number
\begin{equation*}
d=\inf\{\A a+b-1\A,\;\;ab=0,\;\A a \A=\A b\A=1,\;a>0,\;b>0\}
\end{equation*}

 We prove  that $d\in\{0,1\}$. Furthermore $d=1$ if  and only if $A$ has  no  nontrivial  projection.  We define the addition map
  \begin{equation*}
  \phi:\{(a,b)|a>0,b>0,\A a\A=\A b\A=1, ab=0\} \longrightarrow S_{A}
   \end{equation*}
   where $S_{A}$ is the unit sphere of $A$,  with $\phi(a,b)=a+b$. In fact the above statement says that existence of nontrivial projection is equivalent to that $1$ is  an accumulation point of image($\phi$). In fact we have stronger result: in order to prove that a $C^{*}$ algebra has nontrivial projection it is sufficient to prove that there are positive elements $a$ and $b$ in the unit sphere of $A$ with $ab=0$ and $\A a+b-1 \A< 1$.
 It seems that the image of $\phi$ has some useful information about our $C^{*}$ algebra.
In this line we prove that if image($\phi$) is  a finite set, then it is equal to $\{1\}$, and $A$ must be isomorphic to either $C^{2}$ or $M_{2}(\mathbb{C})$.  We also prove that  image($\phi$) is  a closed and connected set for every finite dimensional $C^{*}$ algebra, or  for  a commutative $C^{*}$ algebra in the form $C(X)$ where $X$ is a Peano space. We give examples of $C^{*}$  algebras such that this property fail for corresponding  image($\phi$).\\
 As we will observe "\textit{closed ness of image of $\phi$}" is a hereditary invariant so it can be used as a criterion for non existence of hereditary imbedding of $C^{*}$ algebras.  \\

We also introduce a new concept zero divisor real rank of a $C^{*}$ algebra, as  a zero divisor analogy of real rank theory of $C^{*}$ algebras, see \cite{RIEFFEL}. We observe that this quantity is zero for $A=C(X)$  where $X$ is  a separable compact Hausdorff space or the unit square with the lexicographic topology.\\

  Then, for  an arbitrary $C^{*}$  algebra $A$, we consider the zero divisor graph of positive elements with  connection relation $ab=0$, for  positive  elements $a$ and $b$. This is  a $C^{*}$ algebraic version of the  concept of  zero  divisor graph of  a ring which was  first  introduced in \cite{BECK}.
We will prove that  for  the  Calkin algebra, this  graph is  a  connected  graph which diameter is  equal to 3.\\

\section*{Zero divisors in   $C^{*}$ algebras}
\noindent
Let $A$  be  a $C^{*}$  algebra. We say  that  $A$ has  zero  divisor  if  there are two  non zero  elements  $a$  and $b$  in $A$  such that $ab=0$. In this  case $a$ and $b$ are  called     zero  divisor. An element $x\in A$ is  called  positive if $x=aa^{*}$  for  some $a\in A$.
We  say that  a  $C^{*}$ algebra $A$ is  nontrivial if
$\dim A > 1$. It is  well known that a  nontrivial $C^{*}$ algebra  has  always zero  divisor.
In this note  we  give  a new  proof for this  fact. We  prove  the  following  theorem:
 \begin{thm1}
 Let $A$ be  a nontrivial $C^{*}$ algebra. Then there are nonzero positive elements $a$  and $b$ in $A$ such that  $ab=ba=0$. Furthermore a positive element $a$ is a zero divisor if $0\leq a\leq b$ for some zero divisor $b\in A$.
 If $A$ is a commutative separable $C^{*}$ algebra, then there is a positive element in $A$ which is not a zero divisor.
 \end{thm1}
 Our  main tool to  prove this theorem is the   proposition 1 which is introduced in the introduction .

For a  $C^{*}$  algebra $A$, the  double centralizer of  $A$,  denoted  by $\mathcal{DC}(A)$,  is  defined as follows:\\
$\mathcal{DC}(A)$  is  the  space of all pair $(L,R)$ where  $L$  and $R$  are  bounded  operators on $A$  which  satisfies the relation $R(x)y=xL(y)$  for  all $x$ and $y$ in $A$. Consider the following operations on $\mathcal{DC}(A)$ :

$$
(L_{1},R_{1})+(L_{2},R_{2})=(L_{1}+L_{2},R_{1}+R_{2})$$
$$(L_{1},R_{1})(L_{2},R_{2})=(L_{1}L_{2},R_{2}R_{1})$$
$$
(L,R)^{*}=(R^{*},L^{*}),  \text{where}\; L^{*}(x)=(L(x^{*})^{*}  \text{and}\;   R^{*}(x)=(R(x^{*})^{*}$$
With these operations,  $\mathcal{DC}(A)$ is a  unital $C^{*}$  algebra. It is  the  maximal unitization of  $A$  and it is  isomorphic to $A$  when $A$  is  unital. Furthermore It  contains $A$  as  a $C^{*}$  subalgebra. It can  be easily  shown that  if $(L,R)$ is a double  centralizer and  $L=0$  then $R=0$. For more information on double  centralizers,  see \cite[page 33]{OLSEN}.

\begin{lem1}
If  $A$ has  no zero  divisor  then $\mathcal{DC}(A)$ has  no  zero  divisor, too.
\end{lem1}
\begin{proof}
Assume that $L_{1}L_{2}=0$ and $L_{1}\neq 0$  and $L_{2}\neq 0$. Then  $R_{1}\neq 0$. So there are $x$  and  $y$  in $A$  such that $R_{1}(x)\neq 0$  and  $L_{2}(y)\neq 0$. Thus $R_{1}(x)L_{2}(y)=x L_{1}L_{2}(y)=0$. This  contradicts to the  assumption that  $A$ has  no  zero  divisor.
\end{proof}

Now  we prove  a lemma which we need it in the next corollary and in the proof of proposition 3. This lemma  can be  considered as a zero divisor  version of the Tietze extension theorem;

 \begin{lem2}
 Let $X$ be  a normal topological space and $N$  be  a closed subset of $X$. Assume that $f,g: N\rightarrow [0,\;1]$ be two  continuous functions  which satisfy $fg=0$. Then there are  continuous  functions $F,G :X\rightarrow [0,\;1]$, as extensions of functions $f$  and $g$ respectively, such that $FG=0$
 \end{lem2}
 \begin{proof}
 A  pair of functions $(F,G)$ with values in $[0,\;1]$ satisfies $FG=0$ if  and only if
 $(F(x), G(x))\in \{0\}\times [0,\;1]\cup [0,\;1]\times \{0\}$ but the later space is  homeomorphic to $[0,\;1]$. So the lemma is  a consequence of the  standard Tietze Extension theorem
 \end{proof}

\begin{coro1}
If  $A$  is  a nontrivial unital  commutative $C^{*}$ algebra  then  $A$ has  zero  divisor.
\end{coro1}
\begin{proof}
Since $A$ is commutative and  unital, it  is isomorphic to  $C(X)$, the  space  of    all continuous  complex  valued  functions on a compact Hausdorff  space  $X$. $X$  has  at  least  two  points $x_{1}$ and  $x_{2}$  because $A$ is  not one dimensional algebra. Put
$N=\{x_{1}, x_{2}\}$ and define functions $f$ and $g$ on $N$ with $f(x_{1})=g(x_{2})=1$  and
$f(x_{2})=g(x_{1})=0$. Now the  corollary is  a consequence of the above lemma.
\end{proof}

Using the above two lemmas and proposition 1 we prove theorem 1:

\begin{proof1}

    We show  that $A$  has  zero  divisor.  So there are  nonzero  elements $x$  and  $y$ in  $A$ with $xy=0$. This shows  that for  $a=x^{*}x$  and $b=yy^{*}$ we  have $ab=ba=0$. This  would  complete the  proof of the  theorem.
    Assume  that $A$ is  a nontrivial  $C^{*}$ algebra which does not  have  zero  divisor.  From  lemma  1, we  can  assume  that $A$ is  unital. From the   proposition 1 we  conclude that every  two unital $*$  homomorphism  from $A$ to $A$  are  equal. In  particular for  every unitary  element $u\in A$, an  element  $u\in A$  with $uu^{*}=u^{*}u=1$,  the $*$  homomorphism $\phi_{u}$  with $\phi_{u}(x)=uxu^{*}$ is equal to the identity  homomorphism. Namely $uxu^{*}=x$ for  all  $x\in A$  and  all unitary $u\in A$. This  shows that the space of  unitary  elements lies in the  center of $A$. On the  other  hand,  every unital $C^{*}$  algebra  $A$  can be spanned by its unitary elements,  see \cite [page 25, Theorem I.8.4]{DAVID}.  So  $A$ is  commutative. Now corollary 1 says  that a nontrivial  commutative  $C^{*}$  algebra  has  zero  divisor. This  contradict to our  assumption that $A$  does not  have  zero  divisor.\\
    To  prove  the  second part of the theorem, assume that $b$ is  a positive zero divisor and $0\leq a\leq b$. There is  a positive element $f\in A$ with $bf=fb=0$.   Since $0\leq a\leq b$ we have $0\leq faf=f^{*}af\leq f^{*}bf=0$, by using the formula in \cite[page 20]{OLSEN}. This  shows that $faf=0$. Let $d$ be the positive square root of $a$, that is $d^{2}=a$. Then $faf=0$ implies that $(fd).(fd)^{*}=0$. So $fd=0$ which obviously shows that $fa=fd^{2}=0$. Then $a$ is a positive zero divisor.\\
    Now we prove the last part of the theorem. Since $A$ is separable and commutative, its minimal unitization $\widetilde{A}$ is commutative and separable too. Then, using Gelfand theory, $A  \sim C_{0}(X)$ and $\widetilde{A} \sim C(\widetilde{X})$ for some
    locally compact metric space $X$ and its one point compactification $\widetilde{X}$. Separability of $\widetilde{A}$ implies that $\widetilde{X}$ is metrizable. Let $d$ be a compatible metric for $\widetilde{X}$.
    Put $P_{\infty}=\widetilde{X}-X$, the point at infinity. Now define $f\in A \sim C_{0}(X)$ with $f(x)=d(x,P_{\infty})$. Then $f \in A$ is a positive element which is not a zero divisor. \hspace{8cm}$\Box$

    \end{proof1}
 The following corollary about the  disc algebra  $A(\mathbb{D})$, is an obvious consequence of  the above theorem. Recall that $A(\mathbb{D})$ is the algebra of all holomorphic functions on the unit disc in the plane with continuous extension to the boundary.
 \begin{coro2}
 There is no any norm and any involution on the disc algebra $A(\mathbb{D})$ which make it to a $C^{*}$ algebra.
 \end{coro2}

\begin{coro3}
Let  $H$ be an infinite  dimensional  Hilbert space. Assume  that  $A$  is  a  $C^{*}$  subalgebra  of $B(H)$  such that for  every  nonzero  $T\in A$, $\dim\ker T< \infty$. Then  $A$ is a one  dimensional  algebra.
\end{coro3}

\begin{proof}
If $\dim(A)>1$, by  theorem 1, there are non zero  operators $S,T\in A$  such that $ST=0$. This  shows that range(T)is  a subspace of $\ker T$. So  both $\ker T$  and rang(T)  are  finite  dimensional  space. This  contradict to infinite  dimensionality of $H$. So $\dim (A)=1$.
\end{proof}

\begin{remark3}
Let  $H$ be  an infinite dimensional  Hilbert  space. The  above  corollary shows that  for every non scalar $T\in B(H)$, there  exist  a nonzero operator $S \in C^{*}(T)$, the unital $C^{*}$ algebra generated by $T$,  such that $ \dim \ker(S)= \infty$. Now  a  natural  question is that can we obtain such $S$ by means of  a holomorphic  functional calculus, locally around the spectrum of $T$?  The following example shows that  the answer to this  question is not always  affirmative.
\end{remark3}

\begin{exa1}
Put $H=\ell^{2}$. Let $S \in B(H)$ be the shift operator. Then $\textit{sp}(S)=\overline{\mathbb{D}}$, the closure of the unit disc in the plane, see \cite[page 360]{ACFA}. Assume that $f$ is a holomorphic function defined on a neighborhood of $\overline{\mathbb{D}} $. $f$  has a finite number of zeros on $\overline{\mathbb{D}}$. Then there is a neighborhood of $\overline{\mathbb{D}}$ such that $f$ in the form $f=pg$ where $p$ is  a monic polynomial and $g$ is  a holomorphic map which does not vanish on $\overline{\mathbb{D}}$. Then $g(S)$ is an invertible operator, since $\textit{sp}(g(S))=g(\textit{sp}(S))$ does not contain 0. Furthermore $P(S)$ is an injective operator for it is a composition of operators in the form $S-\lambda$, $\lambda \in \mathbb{C}$ which are obviously injective operators. Then $f(S)=p(S)g(S)$ is an injective operator. Thus for every holomorphic map $f$ defined on a neighborhood of $\textit{sp}(S)$, $\ker f(S)$ is the zero subspace. According to the above remark, the  operator $1-SS^{*}$ is the desired operator, that is a nonzero operator in $C^{*}(S)$ whose kernel has infinite dimension. This simple situation  is  a motivation for  the following question:
\end{exa1}

\begin{question1}
Let $T$ be an operator in $B(H)$. Is there a polynomial $P(x,y)$, in two noncommutative variable $x,y$, such that
$P(T,T^{*})$ is  a non zero operator with infinite dimensional kernel.
\end{question1}

As we mentioned in the introduction there is  a close relation between the projections  and zero divisors. So it is natural to define the addition map $\phi$ as follows and study the properties of its image:
 \begin{equation}\label{EQ}
  \phi:\{(a,b)|a>0,b>0,\A a\A=\A b\A=1, ab=0\}\rightarrow S_{A}
   \end{equation}
 where $S_{A}$ is the unit sphere of $A$ and $\phi(a,b)=a+b$. Put
 \begin{equation}\label{eq}
 d=\inf\{\A a+b-1\A,\;\;ab=0,\;\A a\A=\A b\A=1,\;a>0,\;b>0\}
 \end{equation}

 From now on, for a $C^{*}$ algebra $A$, the above additive map $\phi$ is denoted by $\phi_{A}$.
 In the next proposition we need to the following definitions:\\
 \begin{dfn1}
 A $C^{*}$ subalgebra $B$ of $A$ is hereditary subalgebra if  for every positive element $b \in B$, $a\in A$ with $0\leq a \leq b$ we have $a\in B$, see  \cite[page 13]{DAVID}. An injective $C^{*}$ morphism is called hereditary if its image is a hereditary subalgebra. A  sub $C^{*}$ algebra $B$ of $A$ is called totally non hereditary if there is no a hereditary imbedding
 of $A$ into $B$. Equivalently $B$ is a totally non hereditary sub algebra of $A$, if it is not isomorphic to any hereditary subalgebra of $A$.
 \end{dfn1}

 In the  following proposition $S_{A}^{+}$ and $D_{A}^{+}$ are positive elements of the unit sphere and unit disc
 of $A$ respectively. Obviously  $S_{A}^{+}$ and $D_{A}^{+}$ are closed subsets of $A$.

 \begin{pro3} With the same notations as in the above relations (\ref{EQ}) and (\ref{eq})  and with the assumption that $A$ is a unital  $C^{*}$  algebra we have
 \begin{enumerate}
 \item The image of $\phi_{A}$ is  a subset of $S_{A}$ and $d\in\{0,1\}$. Moreover $d=1$ if  and only if there is no nontrivial projection $a\in A$. \\
  \item Closed ness of image of $\phi_{A}$ is a hereditary invariant.\\
 \item If image of $\phi_{A}$ is a finite set, then it is equal to $\{1\}$ and $A$ is isomorphic to $\mathbb{C}^{2}$ or $M_{2}(\mathbb{C})$.\\
  \item  The image of $\phi_{A}$ is a compact  set if and only if $A$ is a finite dimensional $C^{*}$ algebra.

 \end{enumerate}
 \end{pro3}

 \begin{proof}

 Let $a$ and $b$ be  elements of $A$ with
  \begin{equation}\label{three}
   ab=0,\;\;\A a\A=\A b\A=1.\;\;a>0,\;b>0
   \end{equation}
   Then we have $ba=0$. Thus $C^{*}(a,b)$,the unital $C^{*}$  subalgebra generated by $a$ and $b$ is isomorphic to $C(Y)$ where $Y$is the joint spectrum of $a$ and $b$. Recall that the joint spectrum of $a$ and $b$ is $\{(\psi(a),\psi(b))\mid  \text{ $\psi$ is a character of} \;C^{*}(a,b) \}$. Furthermore $(1,0)$ and $(0,1)$ belong to $Y$.  Because for every norm one positive element $a$ of a commutative $C^{*}$ algebra, there is character $\psi$ such that $\psi(a)=1$.  From $ab=0$ we have $$Y\subseteq X=\{(s,t)\in [0,\;1]^{2}  \mid st=0, s\geq 0, t\geq 0\}$$ By Gelfand theory we replace $a$ by function $s$ and $b$ by function $t$ as elements of $C(Y)$. Now it is obvious that $\parallel a+b\parallel=\parallel s+t \parallel  =1$. So $a+b \in S_{A}$. Assume that $A$ has no projection. Then $Y$ is  a connected space, so $Y=X$. This shows that $\parallel a+b-1\parallel=\parallel s+t-1 \parallel=1$. This means that in a projection less $C^{*}$ algebra $A$ the equation (\ref{three}) implies that $\parallel a+b-1\parallel=1$. Then $d=1$ in a projection less $C^{*}$ algebra.
   On the other hand for a $C^{*}$ algebra $A$ with a projection $a$, we put $b=1-a$. Then for such $a$ and $b$ (\ref{three}) holds with $a+b=1$. This implies that for a $C^{*}$ algebra with nontrivial projection, we have $d=0$. This  completes the proof of (i).\\

    To prove (ii) assume that $B$ is  a hereditary subalgebra of $A$ and image of $\phi_{A}$ is  a closed subset of $A$. We prove that image of $\phi_{B}$ is a closed subset of $B$. Note that image of $\phi_{B}$ is  a subset of image $\phi_{A}$. Assume that $p\in B$  is in  $\overline{\text{image}\; \phi_{B}}$. Then p is in the $\overline{\text{image}\; \phi_{A}}$, too.  Since image $\phi_{A}$ is closed, there are positive elements $a$ and $b$ in $A$, with $ab=0$  and $\A a \A=\A b\A=1,\;\;a+b=p$. So $p \in B$ and $0 \leq a\leq p$ \;$0 \leq b\leq p$. This implies that $a\in B$ and $b \in B$, for $B$ is a hereditary subalgebra of $A$. This shows that $p\in \text{image} \;\phi_{B}$. Thus image $\phi_{B}$ is a closed subset of $B$.\\

 To prove (iii), we assume that image of $\phi$ is a finite  set. We first show that for every self adjoint element $a\in A$, $\textit{sp}(a)$ has at most two points. Then from spectral characterization of finite  dimensional $C^{*}$ algebras described in \cite[4.6.14]{KADRIN}, we conclude that $A$ has finite dimension . On the other hand a finite  dimensional $C^{*}$ algebra is  a direct sum of matric algebras. So the only possibility for $A=\bigoplus M_{n_{i}}(\mathbb{C})$, with the additional property that the spectrum of each self adjoint element has at most two points is either $A\simeq \mathbb{C}^{2}$ or $A\simeq M_{2}(\mathbb{C})$. For  $A=\mathbb{C}^{2}$ there are only two positive zero divisors which norm is equal to one: $(1,\;0)$  and $(0,\;1)$ which sum up to $(1,\;1)=\large 1\small_{\mathbb{C}^{2}}$. For $A=M_{2}(\mathbb{C})$, with a simple geometric interpretation for  algebraic equations (\ref{three}),  we conclude that it implies that
 $a+b=I_{2}$, the identity $2\times 2$ matrix. So image $\phi=\{1\}$ if $A\simeq \mathbb{C}^{2}$ or $A\simeq M_{2}(\mathbb{C})$
To  complete the proof of the proposition we assume that image $\phi$ is a finite set. Then we prove that for a self adjoint element $a\in A$, the $\textit{sp}(a)$  has at most two points. The unital commutative $C^{*}$ subalgebra of $A$  generated by $a$ is isomorphic to $C(\textit{sp}(a))$.  On the other hand  the  finiteness of image$\phi$ is a hereditary  property for a $C^{*}$ algebra $A$. So it is  sufficient to prove that image $\phi$ is an infinite set if $A=C(X)$ where $X$ is a compact topological space with at least three points. Choose    three points $x_{1}, x_{2},x_{3}$ in $X$. Define continuous functions $$f_{\lambda},\;g: \{x_{1}, x_{2},x_{3}\}\longrightarrow  [0,\;1]$$ with $f_{\lambda}(x_{1})=1$,
 $f_{\lambda}(x_{2})=0$,
  $f_{\lambda}(x_{3})=\lambda$  and \\
  $g(x_{1})=g(x_{3})=0,\;\; g(x_{2})=1$\\
  Here $\{x_{1}, x_{2},x_{3}\}$ is equipped to  discrete topology, which is  identical to
  the  subspace  topology inherited from $X$. By lemma 2 there  are  continuous  functions $F_{\lambda},G_{\lambda}$ as extension of $ f_{\lambda}, g$ with $F_{\lambda}G_{\lambda}=0$. Obviously , for each $\lambda \in [0,\;1]$, $F_{\lambda}+G_{\lambda}$ lies in image $\phi$. Furthermore $F_{\lambda}+G_{\lambda}$'s are  different elements  for  different  values of $\lambda$, since they takes  different values at $x_{3}$. This  completes the proof of (ii).\\

  Now we  prove (iv).Let $A$ be a finite dimensional $C^{*}$ algebra .  Assume that $a_{n}$ and $b_{n}$ are two sequences in $A$ which satisfies (\ref{three}) and  $a_{n}+b_{n}$ converges to $z$. Since $A$  is finite  dimensional algebra, the unit  sphere of $A$ is compact. By passing to  subsequences we may assume that $a_{n}$and $b_{n}$ are convergence sequences to elements $a$ and $b$ in $A$. Obviously $a$ and $b$ satisfy  (\ref{three}) and $a+b=z$. Then $z\in \text{image}\;\phi_{A}$. Thus $\text{image}\;\phi_{A}$ is  a closed set.\\
  Now  assume that $\text{image}\;\phi_{A}$ is  a compact set. We  prove that $A$ is  a finite  dimensional  algebra.
  To prove this, we  show that the  spectrum of each self adjoint  element is  a  finite  set. Assume  in contrary that $\textit{sp}(a)$ is  infinite for  a  self  adjoint element $a\in A$. Put $B=C^{*}(a)\sim C(X)$  where $X=\textit{sp}(a)$.  We  show that  there is a   sequence $z_{n}\in  \text{image}\;\phi_{B}\subseteq \text{image}\;\phi_{A}$  such that $\A z_{n} -z_{m} \A = 1$  for  all $n\neq m$. This  obviously  shows that
  $\text{image}\;\phi_{A}$ is  not  compact. Let $\{ x_{1},x_{2},\ldots \}$  be  an infinite subset of $X$. Now  we apply lemma 2 to obtain  a sequence of functions  $f_{n}$  and $g_{n}$ in $C(X)$ which  satisfies (\ref{three})  and the  following  properties:
  \[   \begin{cases}   f_{n}(x_{i})=0\;\; \text{for} \;\;i=1, 2,\ldots,n \\
  g_{n}(x_{i})= 0 \;\;     \text{for} \; i=1,2,\dots n-1 \hspace{1cm}  g_{n}(x_{n})=1   \end{cases}    \]
  Put $z_{n}=f_{n}+g_{n}$. Then $\A z_{n} -z_{m} \A = 1$. This  completes the  proof  of the proposition.

\end{proof}
 The following  corollary is  an obvious consequence of the  above proposition:
 \begin{coro4}
 Assume  that $A$ is  a  projection less $C^{*}$  algebra. Let $a$ and $b$ be two  positive elements  with $ab=0$.
 Then we have $$   \parallel    \parallel a\parallel b \;  + \parallel b\parallel a - \parallel a\parallel \parallel b\parallel   \parallel =\parallel a\parallel \parallel b \parallel $$
 \end{coro4}
In the next example we study the closed ness of image of $\phi_{A}$ for $A=L^{\infty}[0,\;1]$, $B=\ell^{\infty}$, the space of all bounded sequence of complex numbers and $C=\ell^{c}$, the space of all converging sequence of complex numbers.
These three algebra are equipped to the standard norm $\parallel .\parallel_{\infty}$ .

\begin{exa2}
We show  that image of $\phi_{A}$ is $S_{A}^{+}$, the space of norm one positive elements which is  obviously a closed subset of $A$. Let $f$ be an element of  $S_{A}^{+}$. Then we have only two possibility:\\
 i) there are two disjoint subset $Z_{1}$ and $Z_{2}$  of $[0,\;1]$ with non zero measures such that $f(Z_{1})=f(Z_{2})=\{1\}$\\
 ii)there is an strictly increasing sequence of positive numbers $t_{n}$,$t_{0}=0$, which converges to 1 such that the measure of
 $f^{-1}[t_{n},\;t_{n+1})$ is not zero for all n.\\
  In case (i), choose a measurable partition of $[0,\;1]$ by subsets $\widetilde{Z_{1}}$ and $\widetilde{Z_{2}}$, containing $Z_{1}$ and  $Z_{2}$ respectively. Now for $i=1,2$ define:
 \[ f_{i}(x)=\begin{cases} f(x)        &\text{ if $x\in \widetilde{Z_{i}}$},\\
0       &    \text{otherwise} \end{cases} \]
Then $\parallel f_{i}\parallel=1$, $f_{i}\geq 0$,  for $i=1,2$. Moreover $f_{1}f_{2}=0$ and $f=f_{1}+f_{2}$. Hence $f$
belongs to the image of $\phi_{A}$.\\
In case (ii), put $Z_{i}=f^{-1}([t_{i},\;t_{i+1}))$,
 $W_{0}=\bigcup_{i=0}^{\infty} Z_{2i}$,
$W_{1}=\bigcup_{i=0}^{\infty} Z_{2i+1}$\\
Now define \[ f_{i}(x)=\begin{cases} f(x)        &\text{ if $x\in W_{i}$},\\
0       &    \text{otherwise} \end{cases} \]
Then we have again $\parallel f_{i}\parallel=1$, $f_{i}\geq 0$,  for $i=0,1$. Moreover $f_{0}f_{1}=0$ and $f=f_{0}+f_{1}$. Thus $f$
is an element of image of $\phi_{A}$. Then image of $\phi_{A}$ is a closed subset of $A$.\\
Now we consider the $C^{*}$ algebra $B=\ell^{\infty}$.
It can be easily shown that an element $(b_{n})$ in $S_{B}^{+}$ lies in image of $\phi_{B}$ if and only if either $b_{n}$ has a subsequence converging to 1, or $b_{n}=b_{m}=1$, for at least two integers $m$ and $n$. This obviously shows that the image of $\phi_{B}$ is a closed set.\\
Finally we consider   $C=\ell^{c}$. The image of $\phi_{C}$ is the set of all elements $(c_{n})$ of $S_{C}^{+}$  for which
either $c_{n}=c_{m}=1$, for at least two integers $m$ and $n$ or $1=\lim_{n \to \infty} c_{n}=c_{m}$ for some integer $m$. Obviously, this set is not a closed subset of $C$

\end{exa2}
The following corollary is an immediate consequence of the above example and part (ii) of proposition 3:
\begin{coro4}
 $\ell^{c}$ is a totally non hereditary sub algebra of $\ell^{\infty}$.
\end{coro4}

 \begin{remark5}
 As we observe in the  above proposition, to every unital $C^{*}$ algebra $A$,  one associate
 a subset of the unit sphere of $A$, denoted by image $\phi_{A}$.  It seems that image $\phi_{A}$ contains useful information about $C^{*}$  algebraic structure of $A$. The  above proposition is  a motivation to give the next question:
 \end{remark5}
 \begin{question2}
 What  $C^{*}$  algebraic invariant can be extracted from the  topology of image $\phi_{A}$ and its relative position with respect to $\{1\}$. As we see in the part (ii) of proposition 3, the particular case that image $\phi_{A}$ is  a finite set(singleton), gives us a $C^{*}$ algebraic concept. In this case we obtain two possibility for $A$:\;$\mathbb{C}^{2}$ and $M_{2}(\mathbb{C})$. But these two  algebras  has  a common property: Each of them is  a two pints space, $\mathbb{C}^{2}$ is  a  commutative two points space, and $M_{2}(\mathbb{C})$ is  a noncommutative two points space. That is $\mathbb{C}^{2}$ is    the  algebra of  continuous functions on a two point space  $\{a,\;b\}$. On the other hand $M_{2}(\mathbb{C})$  is  a result of a noncommutative  process on a two points space $\{a,\;b\}$ which is equipped with the equivalent relation  $a\sim b$, see \cite[II.2, page 85]{NCG}\\
As we observe in the above proposition the image of $\phi_{A}$ is a compact set when $A$ is a finite dimensional algebra. Now assume that the image of $\phi_{A}$ is a compact set. Does it implies that $A$ is a finite dimensional $C^{*}$ algebra??
 \\

 Note that "image $\phi_{A}$" is  a functor  from the category of $C^{*}$ algebras with injective morphisms to the category of topological spaces. Now with  a method similar to the basic constructions of $K$-theory described in \cite {OLSEN}, we can consider the  inductive limit of this functor with respect to the natural embedding of $M_{n}(A)$ into $M_{n+1}(A)$ which sends a matric $M$ to $M\oplus 0$. With this  construction we obtain a new functor,  which we denote its value at $A$ by $Z(A)$. To what extend this new functor \textit{Z} is useful and worth  studying?

 \end{question2}
 \begin{remark6}
 To prove part (ii) of proposition 3 we used the fact that a $C^{*}$ algebra for which the spectrum of all elements is a finite set must be a finite dimensional $C^{*}$ algebra. In the  following question, we search for a some how generalization of this statement. Of course, in order to obtain a possible  counter example to the following question we have to construct a $C^{*}$ algebra for which all non scalar self adjoint elements have disconnected spectrum:
\end{remark6}

  \begin{question3}
 Let $A$ be  a $C^{*}$ algebra with the property that the topological type of the spectrum of its elements varies in a finite  number of topological space. Does it implies that $A$ has finite dimension?
 \end{question3}
  We  end this section with two questions and  two propositions.
  \begin{question4}
   For  a  finite dimensional  $C^{*}$ algebras a positive  element $a$ is  a zero divisor if and only if $a$ is not invertible. Can this property characterize all finite dimensional $C^{*}$ algebras? In the other word, assume that $A$ is  an infinite dimensional $C^{*}$ algebra. Is there a positive non invertible element $a$ such that $a$ is not a zero divisor. The following proposition is an affirmative answer in some particular case. Can we drop the commutativity of $A$ in this proposition?
\end{question4}

\begin{pro4}
Let $A$ be  a commutative nontrivial separable $C^{*}$ algebra which does not have nontrivial projection. Then there is  a positive element $a\in A$ which is neither invertible nor zero divisor.
\end{pro4}
\begin{proof}
The idea of proof is identical to the method of proof of the last part of theorem 1.
\end{proof}

In the final part of this section we introduce a new concept \textit{zero divisor real rank} of  a $C^{*}$ algebra.
Similar to the idea of noncommutative dimension theory, the real rank of $C^{*}$ algebras which is introduced in \cite{RIEFFEL}, we define the zero divisor real rank of $A$, abbreviated by $ZDRR(A)$ as in the following definition.  Recall that the real rank of $A$ is the least integer $n$ such that every $n+1$ self adjoint elements can be approximate with $n+1$ self adjoint elements $b_{1},\ldots ,b_{n+1}$  such that $\sum_{i=1}^{n+1} b{i}^{2}$ is  invertible.
\begin{dfn2}
We say $ZDRR(A)\leq n$ if  every
$n+1$ tuple of self adjoint elements $(a_{1},\ldots ,a_{n+1})$ can be approximated by an $n+1$-tuple of self adjoint elements $(b_{1},\ldots b_{n+1})$ such that $\sum_{i=1}^{n+1} b_{i}^{2}$ is not a zero divisor.
 \end{dfn2}

 This  definition is well behaved in the sense that $ZDRR(A)\leq n$ implies that $ZDRR(A)\leq m$ for all $m\geq n$. Let's   an $m+1$ tuple
 \begin{equation}\label{RR}
 (a_{1},\ldots a_{n+1},\ldots a_{m+1})
 \end{equation}

is given. We  approximate $(a_{1},\ldots a_{n+1})$ with
 $(b_{1},\ldots ,b_{n+1})$ such that $\sum_{i=1}^{n+1} b_{i}^{2}$ is not a zero divisor. Then $(b_{1},\ldots b_{n+1},a_{n+2},\ldots a_{m+1})$ is  an approximation for (\ref{RR}). Moreover,   an immediate consequence of the second part of theorem 1, is that the expression  $\sum_{i=1}^{n+1} b_{i}^{2} +\sum_{i=n+2}^{m+1} a_{i}^{2}$ is not a zero divisor. So we give the  following  definition:
 \begin{dfn3}
 $ZDRR(A)=n$ if $n$ is the least integer which satisfies $ZDRR(A)\leq n$
 \end{dfn3}
 Obviously $ZDRR(A)$ is not grater than $RR(A)$, the real rank of $A$.\\
 Now  a realization type question is that:
 \begin{question5}
 Let $n$ be  a nonnegative (possibly infinite)  integer. Is there a  $C^{*}$ algebra $A$ such that $ZDRR(A)=n$. The  following proposition give's us a wild class of commutative $C^{*}$ algebras $A$  with $ZDRR(A)=0$:
 \end{question5}
 \begin{pro5}
 Assume that  $A=C(X)$ for a separable compact topological space $X$. Then $ZDRR(A)=0$
 \end{pro5}
 \begin{proof}
 Let $a\in C(X)$  be  a real valued continuous  function. We recall that $b\in A$ is  a zero divisor if and only iff the interior of $b^{-1}(0)$ is not empty.  Now from separability of $X$ we conclude that the  set of all $y \in \mathbb{R}$ such that $b^{-1}(y)$ has non empty interior is  a countable set. So we can choose a sufficiently small
 $\epsilon$  such that $a^{-1}(\epsilon)$ has null interior. Putting $b=a-\epsilon$ we  obtain that $b^{2}$ is not  a
 a zero divisor.

 \end{proof}

 \begin{exa3}
 Assume that $X=\mathbb{I}^{2}$ is the unit square with the lexicographic order topology. $X$ is a compact Hausdorff space but is not a separable space. So it does not satisfies in the conditions of the above proposition. However we prove that for $A=C(X)$, we have $ZDRR(A)=0$. Recall that for every topological space, a function $f\in C(X)$ is a zero divisor if and only if $f^{-1}(0)$ has non empty interior. Now assume that a real valued function $f\in C(X)$ is given. For every positive $\epsilon$, we find a function $f_{\epsilon}$, which is  $\epsilon-$close to f and is not a zero divisor.
 We consider the partition of  $X$ to disjoint union  $X=I_{0}\bigcup I_{1}$ where $I_{0}$ is the interior of $f^{-1}(0)$. Now we decompose $I_{0}=\bigsqcup_{\alpha} J_{\alpha}$ to its connected components. Each $J_{\alpha}$ is in the form $(S_{\alpha},\;T_{\alpha})$ as an interval in the ordered space $X$. Assume that $S_{\alpha}=(x_{1},\;y_{1}),\;\;T_{\alpha}=(x_{2},\;y_{2})$. Choose three continuous  functions (with respect to the standard topology of $\mathbb{R}$) $g_{0}, g_{1}, g_{2}$
 on  the intervals $[y_{1},\;1],\;[x_{1},\;x_{2}],\;[0,\;y_{2}]$,respectively such that $\A g_{i}\A \leq \epsilon$ and each $g_{i}$ vanish only on the endpoints of the corresponding intervals. Define $f_{\alpha}$ on $J_{\alpha}$ as follows:
 \[ f_{\alpha}(x,y)= \begin{cases}   g_{0}(y)       &\text{if $x=x_{1}$ and $y_{1} \leq y \leq 1$},\\
 g_{1}(x)          & \text{if $x_{1} < x < x_{2}$},\\
 g_{2}(y)      & \text{if $x=x_{2}$   and $0 \leq y \leq y_{2}$}.
 \end{cases}
 \]
 \end{exa3}
 Now consider the  function
 \[ f_{\epsilon}(z)= \begin{cases}  f(z)      &\text {if $z\in I_{2}$},\\
 f_{\alpha}(z)        &\text{if $z \in J_{\alpha}$}
 \end{cases}
 \]

 Then $f_{\epsilon}$ is not  a zero divisor and we have $\A f-f_{\epsilon} \A \leq \epsilon$. This  shows that $ZDRR(A)=0$ because every self adjoint element of $A$  is  approximated by a non zero divisor.

\section*{zero divisor  graph  of  positive  elements}

The set of  of  all nonzero positive zero  divisors is denoted  by $Z_{+}(A)$.  We  assign to    $A$, the  undirected  graph  $\Gamma_{+} (A)$ which vertices are $Z_{+}(A)$ such that for distinct vertices a and b there is an edge connecting them if and only if $ab = 0$. Of course this  implies that $ba=0$,  since   positive  elements $ a$ and $b$  are  necessarily   self adjoint. The  graph of projection zero divisors has the same structure as above, but the set of vertices consists of projections in $A$.\\
 A  graph  is  connected  if for  every two elements $a$  and $b$ there  is  a path  of edges  which starts from $a$  and ends to $b$. For  two  vertices $a$ and $b$, we define  $d(a,b)=\inf $ length of  paths  which  connect $a$ to $b$. The  diameter of  a  connected  graph $\Gamma$,  which is  denoted  by $diam \Gamma$  is  $\sup \;\{d(a,b)|\;\;a,b\in \Gamma\}$.\\
Throughout this section,  $H$ is  an infinite  dimensional  separable  Hilbert  space with inner product $<.,.>$.
Recall  that  an  operator  $T \in B(H)$ is  compact if the  closure of  the image of the unit ball of $H$  is  compact.
 The  space of all compact  operators on $H$ is  denoted by $K(H)$ which is a closed two sided ideal in $B(H)$.  The  quotient algebra $B(H)/K(H)$ is called the  Calkin algebra. An operator $\widetilde{T}$ $\in B(H)$  is  a  lift of $T \in B(H)/K(H)$ if $\widetilde{T}$ is mapped to $T$ by the  canonical quotient map  from $B(H)$ to $B(H)/K(H)$.  With these  notations we have  the  following  theorem 2. To prove this theorem we need the following  two lemmas. Lemma 3
 cen be proved with an straightforward application of Grahm Shmith type argument. Lemma 4 is an obvious consequence of definitions. We omit the of both lemma.
\begin{lem3}

Let $H_{1}$ and $H_{2}$ be two infinite dimensional  subspace of a Hilbert space $H$. Then there are infinite dimensional subspaces $L\subseteq H_{1}$  and $F\subseteq H_{2}$ such that $L\perp F$, that is  $<l,f>=0$, for  all $l\in L$ and $f\in F$.
\end{lem3}

\begin{lem4}
Let $H_{0}$  be  an infinite dimensional  subspace of $H$ which orthogonal complement has infinite dimension. Assume that $\pi_{H_{0}}$ is the orthogonal projection on $H_{0}$, then the image of $\pi_{H_{0}}$ in the Calkin algebra is  a nontrivial positive zero divisor.
\end{lem4}

\begin{thm2}
For the  Calkin algebra $A=B(H)/K(H)$ and operator  algebra $B=B(H)$, $\Gamma^{+}(A)$  and $\Gamma^{+}(B)$  are connected graphs which diameter is equal to  3.
\end{thm2}
\begin{proof}
Let $S$  and $T$ be nonzero positive  elements in $A$ which are  zero  divisors. This  means that there are  nonzero positive  elements  $S_{1}$  and  $T_{1}$ in $A$ such that $SS_{1}=TT_{1}=0$. From \cite[proposition 2.3]{AP}, we  obtain nonzero  positive operators $\widetilde{S}$, $\widetilde{S_{1}}$, $\widetilde{T}$  and $\widetilde{T_{1}}$ in $B(H)$ as  lifting of the preceding elements such that
$\widetilde{S}\widetilde{S_{1}}=0$,  $\widetilde{T}\widetilde{T_{1}}=0$  and  non of these 4  later operators are compact, since they are the  lift of nonzero elements in the  Calkin algebra. So rang $\widetilde{S_{1}}\subseteq \ker\widetilde{S} $ and rang $\widetilde{T_{1}}\subseteq \ker \widetilde{T}$. This implies that $\ker\widetilde{S} $ and $\ker \widetilde{T}$ are infinite dimensional space. For
rang $\widetilde{S_{1}}$  and rang $\widetilde{T_{1}}$  are  infinite  dimensional, since they are noncompact operators.\\
By lemma 3, there are infinite dimensional subspaces $L\subseteq \ker \widetilde{S}$  and $F\subseteq \ker \widetilde{T}$ such that $L\perp F$.
Now define two  projection operators $\widetilde{\pi_{L}}$ and $\widetilde{\pi_{F}}$ by orthogonal projection on $L$  and $F$, respectively. Then  $\widetilde{\pi_{L}}$ and $\widetilde{\pi_{F}}$ are positive  and  noncompact operators and  we have $\widetilde{S}\widetilde{\pi_{ L}} = \widetilde{\pi_{L}}\widetilde{ \pi_{F}}=\widetilde{\pi_{F}}\widetilde{T}=0$.\\
 Let $\pi_{L}$ and $\pi_{F}$ be the canonical image of  $\widetilde{\pi_{L}}$ and $\widetilde{\pi_{F}}$ in the Calkin algebra, respectively. Note that $F\subseteq L^{\perp}$ and $L\subseteq F^{\perp}$ so $L^{\perp}$  and $F^{\perp}$ have infinite dimension.
 So Lemma 4 implies that  $\pi_{L}$  and $\pi_{F}$  are nonzero and positive zero divisors in $A$  . Furthermore
$S\pi_{L}=\pi_{L}\pi_{F}=\pi_{F}T=0$. This  shows  that every two elements of $\Gamma^{+}(A)$ can be  connected to each other by a path with at most  three edges. So diam$\Gamma^{+}(A)\leq 3$.\\
 We  shall prove that there are positive elements $T$ and $S$ of $A$,  with $TS\neq 0$,  which are zero divisor but there is no nonzero positive elements $g\in A$ with $gS=gT=0$. This  would  show that $d(T,S)>2$, hence $diam \Gamma_{+}(A)>2$, so the  proof  of the theorem would  be  completed.\\
 For  $H=l^{2}$, define closed subspaces  $H_{0}$, $H_{1}$ and $H_{2}$ as follows:\\
\begin{equation*}
H_{0}=\{(a_{n})\in l^{2}|\;\; a_{n}=0 \text{ when n is odd} \}
\end{equation*}
  \begin{equation*}
  H_{1}=\{(a_{n})\in l^{2}|\;\;a_{n}=0  \text{when $n=4k+2$  for  some integer k}  \}
   \end{equation*}

    \begin{equation*}
    H_{2}=\{(a_{n})\in l^{2}|\;\; a_{n}=0  \text{when n is even} \}
    \end{equation*}

     Each $H_{i}$, for $i=0,1,2$, is  an infinite  dimensional  subspace  with infinite  dimensional orthogonal  complement. Using lemma 4, we conclude that the orthogonal projections on $H_{0}$ and $H_{1}$,   which are denoted by   $\pi_{H_{0}}$,   $\pi_{H{1}}$,    are  positive non compact operators. In fact  their  canonical image in the Calkin algebra are nontrivial positive zero divisors.\\
      Furthermore $H_{0}\perp H_{2}$, $H_{0}\bigoplus H_{2}=H$ so $\large{1}_{H}=\pi_{H_{0}}+\pi_{H_{2}}$ where $\large{1}_{H}$ is the identity operator on $H$. We have  also $H_{2}\subseteq H_{1}$, so $\pi_{H_{1}} \pi_{H_{2}}=\pi_{H_{2}} \pi_{H_{1}}=\pi_{H_{2}}$.\\
       Assume  that for $G\in B(H)$, the operators $G \pi_{H{0}}$  and $G \pi_{H{1}}$ are compact. Then $G.\pi_{H_{2}}=G.\pi_{H_{1}}\pi_{H_{2}}$ is  a compact operator since the  space of  compact operators is  an ideal. Then $G=G.\large{1}_{H}=G.\pi_{H_{0}}+G.\pi_{H_{2}}$ is  a compact operator.\\
       On the  other hand
$\pi_{H_{0}} \pi_{H_{1}}$ is  not  a compact operator  because it is the projection on the  infinite  dimensional  space  $H_{0}\bigcap H_{1}$.  Let $S$  and $T$ be the  image of $\pi_{H_{0}}$ and $\pi_{H_{1}}$ in the Calkin algebra respectively. Then $S$ and $T$ are two distinct  nonzero positive  zero divisors in $\Gamma^{+}(A)$ such that there is no a path in $\Gamma^{+}(A)$ with one or two edges connecting $T$ to $S$. Thus $d(S,T)>2$ hence $diam \Gamma^{+}(A)>2$. This  completes the proof of the theorem for the Calkin algebra. Obviously the same argument as above can be applied to prove the theorem
for the operator algebra $B(H)$
\end{proof}

In the following corollary, we try to extend   theorem 2 to factors. Recall that a Von Neumann algebra  is  a $C^{*}$ subalgebra of $B(H)$ which contains the identity operator and is  closed with respect to weak operator topology.  A Von Neumann algebra with trivial center is called a factor.
It is  well known that for every operator $T$ in a Von Neumann algebra $A$, the projection onto the kernel and the closure of the rang of $T$ belong to $A$, see \cite[page 19]{OLSEN}. On the other hand the main objects which are used in the proof of the above theorem are projections. So we naturally obtain the following extension of theorem 2:

\begin{coro}
If $A$ is  a factor with $\dim A \geq 3$, then $\Gamma^{+}(A)$ is a connected graph  which diameter is not grater than 6.

\end{coro}

\begin{proof}
We  assume that $A$ is  not isomorphic to any operator algebra $B(H)$, since we have already done it in the  above theorem. Let $A\subseteq B(H)$ be  a factor and $S$ and $T$ be two positive zero divisors in $A$ with kernels $H_{S}$ and $H_{T}$ respectively.  Since $A$ is  a von Neumann algebra, the orthogonal projections $\pi_{S}$ on $H_{S}$, $\pi_{T}$ on $H_{T}$, $S\bigwedge T $ on $H_{T}\cap H_{S}$ and  $S\bigvee T$ on $H_{S}+H_{T}$ belong to $A$.
 If $S \bigvee T$ is  not the identity operator, then $S,\;\pi_{S}, 1-S\bigvee T,\; \pi_{T},\;T$ is a path from $S$ to $T$ with length 4.

 If $H_{T}\cap H_{S}$ is  not the zero subspace then $T,\; \pi_{T},\; 1-\pi_{T},\; S\bigwedge T,\; 1-\pi_{S},\; \pi_{S},\; S$ is  a path  in $\Gamma^{+}(A)$ with length 6 which connects $S$ to $T$.
 Now assume that $H_{S}+H_{T}=H$ and $H_{T}\cap H_{S}=\{0\}$. Since $A$ is  a factor which is not isomorphic to a full operator  algebra $B(H)$, $A$ has no minimal projection, see \cite[5.1.7]{KADRIN}. So there is  a proper subspace $L$
 of  $H_{S}$ such that $\pi_{L} \in A$. Then $S,\;\pi_{L}, 1-\pi_{L}\bigvee T,\; \pi_{T},\;T$ is  a path which connects $S$ to$T$. This  completes the proof of the corollary.

\end{proof}

The  following proposition shows that in the  above  corollary  the assumption $\dim A \geq 3$ can not be dropped.

\begin{pro6}
$\Gamma^{+}(A)$  is  a  disconnected  graph  for $A=M_{2}(\mathbb{C}) $
\end{pro6}
\begin{proof}
Every  positive matrices is unitary equivalent to  a diagonal matrices with positive  entries. On the  other hand  a zero divisor in matrix algebra is  necessarily a singular  matrices. This  shows that a positive zero divisor  $a\in M_{2}(\mathbb{C})$  is  in the form $a =\lambda p$, where $p=p^{*}=p^{2}$ is  a  projection  and $\lambda$ is  a positive real number. So it is  sufficient to show that the graph of projection zero divisors is not  connected.
The geometric interpretation of a projection $p\in M_{2}(\mathbb{C})$ is the orthogonal projection onto  a complex line $L_{p}$ in $\mathbb{C}^{2}$. So $pq=0$ for  projections $p$  and $q$  if and only if $L_{p}\perp L_{q}$. This  shows that two projections $p$  and  $q$ are  connected to each other, as two vertices of the  graph of projection zero divisors if and only if $L_{p}=L_{q}$ or $L_{p}\perp L_{q}$. This obviously  shows that the  graph of projection zero  divisors is not  connected. Thus $\Gamma^{+}(A)$ is a disconnected  graph for $A=M_{2}(\mathbb{C})$.
\end{proof}

\textbf{Final remarks  and  Questions}\\
It is  interesting to compare theorem 2  with  a pure algebraic result in \cite[Theorem 2.3]{AL}, which assert that the zero divisor graph  of  a  commutative ring is connected and its diameter is not greater than 3. However, as proposition 6 above shows, the  method of proof in \cite{AL} obviously can not be applied to positive zero divisor graph of $C^{*}$ algebras. So for  a $C^{*}$ algebra $A$, "the positive zero divisor  graph"  has a different nature  from the
 graph of  zero divisors of $A$, as a pure algebraic ring.\\
  Then it is interesting to search for some counter examples which is stated in the following question:

\begin{question6}

Can one  construct an infinite dimensional $C^{*}$  algebra $A$ which satisfies in one of the following  conditions?

\begin{enumerate}

\item $\Gamma^{+}(A)$  is  connected and diam$\Gamma^{+}(A)=\infty$\\

\item $\Gamma^{+}(A)$  is  connected and  $3<$diam$\Gamma^{+}(A)<\infty$\\

\item $\Gamma^{+}(A)$  is  a disconnected graph

\end{enumerate}

In the  proof of theorem 2, we  essentially  used projections to  construct  a path between two positive zero  divisors.
So each part of this  question seems to be interesting for   $C^{*}$ algebras  without nontrivial projections.
\end{question6}

\begin{remark7}
 Note  that the  last part of the theorem 1 is  no longer true if we drop  the  assumption of separability of $C^{*}$  algebra $A$. As   a  counter example, put $A=C_{0}(X)$ where $X$ is the  Alexanderof line  , or  the  long  line. Recall that the  long line  is  constructed as  follows:
 Let $J$ be  the  first uncountable ordinal number.  Namely $(J,\preceq)$  is  a well ordered set which is  not  a  countable  set but  for every $j_{0}\in J$, the  set $\{j\in J \mid j\prec j_{0}\}$ is  a  countable  set. Now put
  $X= J \times [0,\;1)$ and  equip it with the lexicographic order topology. We  show that for      a  sequence of  compact subsets  $K_ {n}$ of $X$ the  interior of $X\setminus \cup_{n=1}^{\infty} K_ {n}$ is  not  empty. For each $\alpha \in J$ define $\widetilde {\alpha}=(\alpha,0) \in X$. The  first  element of $X$ is  denoted  by $\widetilde{0}$.
  Then $\cup_{\alpha \in J}  [\widetilde{0},\;\widetilde{\alpha})$ is  an open cover for $X$, so it  has a countable open subcover $\cup_{n=1}^{\infty} [\widetilde{0},\;\widetilde{\alpha}_{n})$ for  $\cup_{n=1}^{\infty} K_{n}$ .  Now  choose an element  $\beta \in J$
 such that $\beta \succ \alpha_{n}$ for   all natural number n. Then the  open set $\{x\in X \mid x\succ \widetilde{\beta}\}$ is   contained in  $X\setminus \cup_{n=1}^{\infty} K_ {n}$.
 This  shows that for  each element $f\in C_{0}(X)$, the  interior of $f^{-1}(0)$ is  a  non empty set hence f is  a  zero divisor. Because for  each  compact  subset $K$  of  $\mathbb{C}-\{0\}$, $f^{-1}(K)$ is a  compact  set.
 Write $\mathbb{C}-\{0\}$  as   a countable  union of compact sets  $K_{n}$. So the interior of $f^{-1}(0)=X\setminus \cup  f^{-1}(K_{n})$  is  not empty.So   in the  last part of theorem 1,  separability of $A$ is a  necessary  condition.  Is  commutativity of  $A$  a  necessary  assumption too?

\end{remark7}

We end the paper with the  following  question
    which is  about the zero  divisors of the inductive  limit of a  directed  system of  $C^{*}$ algebras:

\begin{question7}
Let  $A_{n}$'s  be  a  directed  system of $C^{*}$  algebras. Can one  approximate a  zero  divisor of $\underrightarrow{A_{n}}$  with a zero  divisor of some $A_{n}$. This  question can be  considered  as an approximation type  question in the area of inductive limit of $C^{*}$  algebras see \cite[page 303]{OLSEN} .   How can we  compare the positive  zero  divisor graph of $\underrightarrow{A_{n}}$  with the corresponding  graph of $A_{n}$'s?
\end{question7}

\end{document}